\documentclass[10pt]{amsart}

\usepackage{amssymb, amsmath, amsfonts}
\usepackage{amscd}
\usepackage{verbatim}
\usepackage{color}
\usepackage[color,all]{xy}
\usepackage[all]{xy}

\numberwithin{equation}{section}

\newcommand{\pp}{\mathbb P}
\newcommand{\cc}{\mathbb C}
\newcommand{\cW}{\mathcal W}
\newcommand{\cQ}{\mathcal Q}
\newcommand{\cT}{\mathcal T}
\newcommand{\cP}{\mathcal P}
\newcommand{\Sym}{\mathrm{Sym}}
\newcommand{\rank}{\mathrm{rk}\,}
\newcommand{\Hom}{\mathrm{Hom}}
\newcommand{\Ker}{\mathrm{Ker}}
\newcommand{\Gr}{\mathrm{Gr}}

\newcommand{\Prin}{\mathrm{Prin}}
\newcommand{\cF}{\mathcal{F}}
\newcommand{\tE}{\tilde{E}}
\newcommand{\Elm}{\mathrm{Elm}}
\newcommand{\LQ}{LQ_{-e}(W)}
\newcommand{\LQo}{LQ_{-e}(W)^\circ}
\newcommand{\wj}{{\widetilde{j}}}

\newcommand{\cO}{\mathcal{O}}
\newcommand{\Kc}{K_{C}}
\newcommand{\Quot}{\mathrm{Quot}}

\newcommand{\bE}{\overline{E}}
\newcommand{\Indet}{\mathrm{Indet}}

\newcommand{\Aut}{\mathrm{Aut}\,}
\newcommand{\Image}{\mathrm{Im}\,}
\newcommand{\ev}{\mathrm{ev}}

\newcommand{\Oc}{{\mathcal O}_{C}}
\newcommand{\sRat}{\underline{\mathrm{Rat}}\,}
\newcommand{\Rat}{\mathrm{Rat}\,}
\newcommand{\sPrin}{\mathrm{\underline{Prin}}\,}

\newcommand{\LG}{\mathrm{LG}}
\newcommand{\Supp}{\mathrm{Supp}}
\newcommand{\qf}{Q_F^e}
\newcommand{\qfc}{(Q_F^e )^\circ}
\newcommand{\be}{\bar{e}}
\newcommand{\cE}{{\mathcal E}}

\newcommand{\isom}{\xrightarrow{\sim}}

\newcommand{\Cd}{C^{(d)}}
\newcommand{\tY}{\widehat{Y}}
\newcommand{\wcE}{\widetilde{\cE}}

\usepackage{hyperref}
\usepackage{color}
\usepackage[color,all]{xy}
\hypersetup{colorlinks, citecolor=red, linkcolor=blue}

\newtheorem{theorem}{{\textbf Theorem}}[section]
\newtheorem{prop}[theorem]{{\textbf Proposition}}
\newtheorem{cor}[theorem]{{\textbf Corollary}}
\newtheorem{lemma}[theorem]{{\textbf Lemma}}

\newtheorem{defn}[theorem]{{\textbf Definition}}

\newtheorem{rmk}[theorem]{{\textbf Remark}}
\newenvironment{remark}{\begin{rmk}\rm}{\end{rmk}}
\newenvironment{definition}{\begin{defn}\rm}{\end{defn}}

\title[Irreducibility of Lagrangian Quot schemes]{Irreducibility of Lagrangian Quot schemes \\over an algebraic curve}

\author{Daewoong Cheong}
\address{Chungbuk National University, Department of Mathematics, Chungdae-ro 1, Seowon-Gu, Cheongju City, Chungbuk 28644, Korea}
\email{daewoongc@chungbuk.ac.kr}

\author{Insong Choe}
\address{Department of Mathematics, Konkuk University, 1 Hwayang-dong, Gwangjin-Gu, Seoul 143-701, Korea}
\email{ischoe@konkuk.ac.kr}

\author{George H.\ Hitching}
\address{Oslo Metropolitan University, Postboks 4, St. Olavs plass, 0130 Oslo, Norway}
\email{gehahi@oslomet.no}

\begin{document}

\begin{abstract} Let $C$ be a complex projective smooth curve and $W$ a symplectic vector bundle of rank $2n$ over $C$. The Lagrangian Quot scheme $\LQ$ parameterizes subsheaves of rank $n$ and degree $-e$ which are isotropic with respect to the symplectic form. We prove that $\LQ$ is irreducible and generically smooth of the expected dimension for all large $e$, and that a generic element is saturated and stable. \end{abstract}

\maketitle

\section{Introduction}

Let $C$ be a smooth algebraic curve of genus $g \ge 0$ over $\cc$. A vector bundle $W$ over $C$ is called \textsl{symplectic} if there exists a nondegenerate skew-symmetric bilinear form $\omega \colon W \otimes W \to L$ for some line bundle $L$. Such an $\omega$ is called an $L$-valued symplectic form. A subsheaf $E$ of $W$ is called \textsl{isotropic} if $\omega|_{E \otimes E} = 0$. By linear algebra, a symplectic bundle has even rank $2n$ and any isotropic subsheaf has rank at most $n$. An isotropic subbundle (resp., subsheaf) of rank $n$ is called a \textsl{Lagrangian subbundle} (resp., \textsl{Lagrangian subsheaf}). For information on semistability and moduli of symplectic bundles, see \cite{BG}.

For vector bundles,  Popa and Roth proved the following result on the irreducibility of Quot schemes.

\begin{theorem} \label{PRirr} (\cite[Theorem 6.4]{PR}) \ 
For any vector bundle $V$ over $C$, there is an integer $d(V, k)$ such that for all $d \ge d(V,k)$, the Quot scheme $Q^{k,d}(V)$ of quotient sheaves of $V$ of rank $k$ and degree $d$ is irreducible.
\end{theorem}

As a corollary, they showed that for sufficiently large $d$, the Quot scheme $Q^{k,d}(V)$ is generically smooth of the expected dimension, and a general point of $Q^{k,d}(V)$ corresponds to an extension $0 \to E \to V \to V/E \to 0$ where $E$ and $V/E$ are stable vector bundles. A significant feature of this theorem is that it holds for an arbitrary bundle $V$, with no assumption of generality or semistability.

The main goal of the present paper is to obtain the analogous result for Lagrangian Quot schemes of symplectic bundles (Theorem \ref{thmA}). One may try to argue as in \cite{PR}, but one vital step does not appear to adapt in an obvious way: Given a symplectic bundle $W$ of rank $2n$ and for a fixed vector bundle $E$ of rank $n$, the space parameterizing Lagrangian subsheaves $E \subset W$ is a locally closed subset of $\pp H^0(C, \Hom(E, W))$, whose irreducibility seems difficult to decide. This is discussed further at the beginning of {\S} \ref{irred}.

We take instead a different approach: We exploit the geometry of symplectic extensions, together with deformation arguments, as developed in \cite{CH2} and \cite{Hit1}. In particular, Proposition \ref{QFcDense} uses a geometric interpretation for the statement that a nonsaturated Lagrangian subsheaf can be deformed to a subbundle. The connection between extensions and geometry is provided by principal parts, and is developed in {\S} \ref{SymplecticExts}. This provides an alternative language to \v{C}ech cohomology for bundle extensions over curves, and makes transparent the link between the geometric and cohomological properties of the extensions.

We remark that the same argument applies to the vector bundle case, and we expect that similar results can be obtained by these methods for other principal bundles.

We expect that the main result in this paper can be applied to solve the problem on counting maximal Lagrangian subbundles of symplectic bundles, as Holla \cite{Ho} used the irreducibility of Quot schemes to count maximal subbundles of vector bundles. 
Also we expect that an effective version of the irreducibility result for semistable bundles would yield an effective  base freeness (or very ampleness) result on the generalized theta divisors on the moduli of symplectic bundles, as in \cite[{\S} 8]{PR} for vector bundles.
We note that Theorem \ref{thmA} does not give an effective bound on $e$ but only the existence of a bound, due to the existence statement in Proposition \ref{evSurj}. It would be nice to have an effective and reasonably small uniform bound for semistable symplectic bundles.

\subsection*{Acknowledgements}
Our very sincere thanks to the referee of an earlier version for an extremely careful reading, and for detailed and informative feedback which greatly improved the paper, as well as our own understanding.

The first and second authors were supported by Basic Science Research Programs through the National Research Foundation of Korea (NRF) funded by the Ministry of Education (NRF-2016R1A6A3A11930321 and NRF-2017R1D1A1B03034277 respectively). The third author sincerely thanks Konkuk University, Hanyang University and the Korea Institute for Advanced Study for financial support and hospitality.

\subsection*{Notation} Throughout, $C$ denotes a complex projective smooth curve of genus $g \ge 0$. If $W$ is a vector bundle over $C$ and $E \subset W$ a locally free subsheaf, we denote by $\bE$ the saturation, which is a vector subbundle of $W$.

\section{Lagrangian Quot schemes}

In this section, we define the Lagrangian Quot scheme of a symplectic bundle and study its tangent spaces. 

Given a vector bundle $V$ over $C$, the Quot scheme $Q^{k,d}(V)$  parameterizes quotient sheaves of $V$ of rank $k$ and degree $d$; alternatively, subsheaves of $V$ of rank $\rank V - k$ and degree $\deg V - d$. Let $W$ be a bundle of rank $2n$ which carries an $L$-valued symplectic form, where $\deg L = \ell$. Then from the induced isomorphism $W \cong W^* \otimes L$, we have $\deg W = n\ell$. We denote by $\LQ$ the sublocus of $Q^{n, e + n \ell} (W)$ consisting of Lagrangian subsheaves of degree $-e$ and call it a \textsl{Lagrangian Quot scheme}. 

\begin{remark} \label{uniqueness} Note that $\LQ \hookrightarrow Q^{n, e + n\ell}(W)$ depends on the choice of symplectic form $\omega$. However, by \cite[Remarque, p.\ 130]{G}, if $\omega$ and $\omega'$ are two symplectic forms on $W$ then there exists a bundle automorphism $\iota$ of $W$ such that $\iota^* \omega' = \omega$. Then $F \mapsto \iota(F)$ induces an isomorphism $LQ_{-e} (W, \omega) \isom LQ_{-e} (W, \omega')$. In view of this, we shall abuse notation and write simply $\LQ$. \end{remark}

We recall some other important notions: For each integer $e$ and each $x \in C$ we have the \textsl{evaluation map} $\ev_x^e \colon Q^{n, e + n\ell} (W) \dashrightarrow \Gr ( n, W|_x )$ which sends a subsheaf $E$ to the fiber $E|_x$, when this is defined. Also, let $\LG(W)$ be the \textsl{Lagrangian Grassmannian bundle} of $W$, that is, the subfibration of $\Gr(n, W)$ whose fiber at $x \in C$ is the Lagrangian Grassmannian $\LG( W|_x )$.

\begin{lemma} \label{LQproperties} \ Let $W$ be an $L$-valued symplectic bundle of rank $2n$ as above. 
If $g \ge 2$ and $e \ge  \frac{n(g - 1 - \ell)}{2} $, then the locus $\LQ$ is  a nonempty  
closed subset of $Q^{n, e + n\ell} (W)$.
 \end{lemma}

\begin{proof} By \cite[Theorem 1.4 and Remark 3.6]{CH2}, any symplectic bundle has a Lagrangian subbundle of degree $-e_0$ for some $e_0 \le \left\lceil   \frac{n(g - 1 - \ell)}{2} \right\rceil$. For $e > e_0$, we can take an elementary transformation of the Lagrangian subbundle of degree $-e_0$ to get a Lagrangian subsheaf of degree $-e$. This proves the nonemptyness.

For the closedness: Write $\Indet(\ev_x^e)$ for the indeterminacy locus of $\ev_x^e$:
\[ \Indet(\ev_x^e) \ = \ \{ [E \to W] \in Q^{n, e + n\ell} (W) : E \hbox{ is not saturated at } x \} , \]
which is a closed subset of $Q^{n, e + n\ell}(W)$. It is easy to see that
\[ \LQ \ = \ \bigcap_{x \in C} \left( (\ev_x^e)^{-1} \left( \LG ( W|_x ) \right) \cup \Indet ( \ev_x^e ) \right) . \]
As $\LG(W|_x)$ is closed in $\Gr (n , W|_x)$, we see that $\LQ$ is closed. \end{proof}

\begin{remark}
The genus assumption $g \ge 2$ is imposed to get the sharp bound $e \ge  \frac{n(g - 1 - \ell)}{2} $ for nonemptyness of $\LQ$. This bound is proven in \cite{CH2} for $g \ge 2$, but for the case $g=0$ or $1$,  we still have an existence of a bound to guarantee the nonemptyness of $\LQ$.
\end{remark}

We denote by $\LQo$ the open sublocus of $\LQ$ corresponding to vector bundle quotients. The following is  a generalization of \cite[Lemma 4.3]{CH3}.

\begin{prop} \label{Zariski} \ Assume that $\LQo$ is nonempty. Let $[j \colon E \to W]$  be a point of $\LQo$.
\begin{enumerate}
\item[(a)] Every irreducible component of $\LQo$ has dimension at least
\[
\chi(C, L \otimes   \Sym^2 E^* ) 
 \ = \  \frac{n(n+1)}{2} (\ell -g+1) + (n+1)e.
\]
\item[(b)] The Zariski tangent space of  $\LQo$  at $[j: E \to W]$  is given by
\[ T_j \LQo \ \cong \ H^0 ( C, L \otimes \Sym^2 E^* ) . \]  
\item[(c)] If $h^1( C, L \otimes   \Sym^2 E^*  ) =0$, then  $\LQo$ is smooth and of dimension $\chi(C, L \otimes   \Sym^2 E^* ) $ at $j$.
\end{enumerate} \end{prop}

\begin{proof} (a) Let $Z$ be an irreducible component of $\LQo$. Let $[j \colon E \to W]$  be a point of   $Z$ contained in no other  component. Let $\sigma \colon C \to \LG(W)$ be the section corresponding to the subbundle $[j \colon  E  \to W]$. Let $P$ be the Hilbert polynomial of the subscheme $\sigma(C)$ of $\LG(W)$ and $Y$ a component of ${\mathrm{Hilb}}^P(\LG(W))$ containing the point $[\sigma(C)]$. The normal bundle of $\sigma(C)$ in $\LG(W)$ is isomorphic to the restriction of the vertical tangent bundle $T_\pi = \Ker(d \pi)$, which in turn is isomorphic to $L \otimes \Sym^2 E^*$. Hence by the deformation theory of Hilbert schemes, we have
\[
\dim_{[\sigma(C)]} Y \ \ge \ \chi(C, L \otimes \Sym^2 E^*).
\]
Since a general member of $Y$ corresponds to a section of $\pi$, there is a rational map $Y \dashrightarrow \LQo$ defined on a nonempty open subset. As $[\sigma(C)]$ is mapped to $j$, the image of $Y$ lies inside $Z$. Clearly the map $Y \dashrightarrow Z$ is generically injective, so we see that $\dim Z \ge \chi(C, L \otimes \Sym^2 E^*)$.

(b) Let $\alpha \colon E \to W/E \cong E^* \otimes L$ represent a tangent vector to the  Quot scheme $Q^{n, e+n\ell}(W)$ at $[j \colon E \to W]$. For each $x \in C$, the section $\alpha$ defines an element $\alpha(x) \in T_{j(E|_x)} \Gr(n, W|_x)$, and the deformation preserves isotropy of $E$ if and only if $\alpha(x)$ is tangent to the subvariety $\LG (W|_x) \subset \Gr (n, W|_x)$ for all $x$.

The result now follows from the following description of the tangent space of the Lagrangian Grassmannian:
\[
T_{j(E|_x)} \LG (W|_x)  =  (L \otimes \Sym^2 E^*)|_x \ \subset \ (L \otimes  E^* \otimes E^*)|_x  =  T_{j(E|_x)} \Gr (n, W|_x) .
\]

(c) By (a) and (b), if $h^1(C, L \otimes \Sym^2 E^*) = 0$ then 
\[
\dim T_j \LQo \ = \ \chi(C, L \otimes \Sym^2 E^*)  \ \le \ \dim_j \LQo.
\] 
Thus we have equality and $\LQo$ is smooth at $j$.
\end{proof}

\section{Symplectic extensions} \label{SymplecticExts}

If $F$ is a Lagrangian subbundle of a symplectic bundle $W$, then there is a short exact sequence $0 \to F \to W \to F^* \otimes L \to 0$. An extension induced by a symplectic structure in this way will be called a \textsl{symplectic extension}. In this section, we recall or prove some facts on symplectic extensions which we will need later.


Recall that any locally free sheaf $V$ on $C$ has a flasque resolution
\[ 0 \ \to \ V \ \to \ \sRat (V) \ \to \ \sPrin (V) \ \to \ 0 , \]
where $\sRat (V) = V \otimes_{\Oc} \sRat (\Oc)$ is the sheaf of sections of $V$ with finitely many poles, and $\sPrin (V) = \sRat (V)/V$ is the sheaf of principal parts with values in $V$. Taking global sections, we have a sequence of Abelian groups
\begin{equation} 0 \ \to \ H^0 (C, V) \ \to \ \Rat (V) \ \to \ \Prin (V) \ \to \ H^1 (C, V) \ \to \ 0 . \label{cohomseq} \end{equation}
The \textsl{support} of a principal part $p \in \Prin(V)$ is the finite set of points $x \in C$ where $p_x \in \sRat (V)_x / V_x$ is nontrivial. 
 Thus $p$ is a global section of $V(D) / V$ for some effective divisor $D = \Sigma_{i=1}^m d_i x_i$ on $C$. Abusing notation, $p$ can be represented by a finite sum
\[ \frac{v_1}{z_1^{d_1}} + \cdots + \frac{v_m}{z_m^{d_m}} \]
where $z_i$ is a local parameter at $x_i$ for $1 \le i \le m$, and $v_i$ is a local section of $V$ near $x_i$. Note that the principal part $p$ is determined by the images of the $v_i$ in $V_{x_i} / \left( m_{x_i}^{d_i} \cdot V_{x_i} \right)$.

For $\beta \in \Rat (V)$, we denote by $\overline{\beta}$ the principal part $\beta \mod H^0(C,V)$. If $p \in \Prin(V)$, we write $[ p ]$ for the associated class in $H^1 ( C, V)$.

\subsection{Symmetric principal parts and symplectic extensions} \label{SymmPptSympExt}

Let $F$ be any bundle of rank $n$. The sheaf $L^{-1} \otimes F \otimes F$ has a natural involution exchanging the factors of $F$. The \textsl{transpose} $^tp$ of a principal part $p \in \Prin( L^{-1} \otimes F \otimes F )$ 
 is defined to be the image of $p$ under this involution. Then $p$ is \textsl{symmetric} if $^tp = p$, or equivalently $p \in \Prin(L^{-1} \otimes \Sym^2 F)$. Note that this is stronger than the condition $[ ^tp ] = [ p ]$ in $H^1 ( C, L^{-1} \otimes F \otimes F )$.

Now any $p \in \Prin ( L^{-1} \otimes F \otimes F)$ defines naturally an $\Oc$-module map $F^* \otimes L \to \sPrin (F)$, which we also denote $p$. Suppose $p$ is a symmetric principal part in $\Prin ( L^{-1} \otimes \Sym^2 F )$. Following \cite[Chapter 6]{K}, we define a sheaf $W_p$ by
\begin{equation} W_p (U) \ := \ \{ (f, \varphi) \in \sRat(F)(U) \oplus ( F^* \otimes L )(U) : \overline{f} = p (\varphi) \} . \label{W} \end{equation}
for each open set $U \subseteq C$. It is not hard to see that this is an extension of $F^* \otimes L$ by $F$.

Now there is a canonical pairing $\langle \: , \rangle \colon \sRat (F) \oplus \sRat ( F^* \otimes L ) \to \sRat (L)$. By an easy computation (see the proof of \cite[Criterion 2.1]{Hit1} for a more general case), the standard symplectic form on $\sRat(F) \oplus \sRat(F^* \otimes L)$ defined on sections by
\begin{equation} \omega \left( ( f_1 , \phi_1 ) , ( f_2 , \phi_2 ) \right) \ = \ \langle f_2 , \phi_1 \rangle - \langle f_1 , \phi_2 \rangle \label{standardsympform} \end{equation}
restricts to a \emph{regular} symplectic form on $W_p$ with respect to which the subsheaf $F$ is Lagrangian. This shows that for each  symmetric principal part $p \in  \Prin(L^{-1} \otimes \Sym^2 F)$ there is a naturally associated symplectic extension of  $F^* \otimes L$ by $F$. We now give a refinement of \cite[Criterion 2.1]{Hit1}, showing that every symplectic extension can be put into this form.

\begin{lemma} \label{sympext} Let $W$ be any symplectic bundle and $F \subset W$ a Lagrangian subbundle.
\begin{enumerate}
\item[(a)] There is an isomorphism of symplectic bundles $\iota \colon W \isom W_p$ for some symmetric principal part $p \in \Prin (L^{-1} \otimes \Sym^2 F)$ such that $\iota(F)$ is the natural copy $W_p \cap \sRat (F)$ of $F$ in $W_p$, which is given over each open set $U \subseteq C$ by $\{ ( f, 0 ) : f \in F(U) \}$.
\item[(b)] The class of the extension $0 \to F \to W_p \to F^* \otimes L \to 0$ in $ H^1 ( C, L^{-1} \otimes \Sym^2 F)$ coincides with $[p]$.
\end{enumerate} \end{lemma}

\begin{proof} (a) As much of this proof is computational, we outline the main steps and leave the details to the interested reader.

Since $F$ is isotropic, $W$ is an extension $0 \to F \to W \to F^* \otimes L \to 0$. By \cite[Lemma 3.1]{Hit0}\footnote{This is unpublished, but it is the obvious generalization of the rank two case treated in \cite[Lemma 6.5]{K}.}, there exists ${p'} \in \Prin ( L^{-1} \otimes F \otimes F )$ such that the sheaf of sections of $W$ is given by
\begin{equation} U \ \mapsto \ W_{p'} (U) \ = \ \left\{ ( f, \phi ) \in \sRat(F)(U) \oplus (F^* \otimes L)(U) : p' ( \phi ) = \overline{f} \right\} . \label{Wpp} \end{equation}

Using the facts that $F$ is isotropic and the form is antisymmetric and nondegenerate, one shows that there exist $A \in \Aut ( F )$ and $B \in \Rat( L^{-1} \otimes \wedge^2 F )$ such that the given symplectic form $\omega'$ on the sheaf $W_{p'}$ is given by
\begin{equation} \omega' ( ( f_1 , \phi_1 ) , (f_2 , \phi_2 ) ) \ = \ \langle A(f_2) , \phi_1 \rangle - \langle A(f_1) , (\phi_2) \rangle + \langle B(\phi_2) , \phi_1 \rangle \label{pairing} \end{equation}

Using in addition that the restriction of $\omega'$ to $W_{p'}$ is regular, one shows that
\[ A{p'} - {^{^t}\!(A{p'})} + \overline{B} \ = \ \left( A{p'} + \frac{\overline{B}}{2} \right) - {^{^t}\!\left( A{p'} + \frac{\overline{B}}{2} \right)} \ = \ 0 \ \in \ \Prin ( L^{-1} \otimes F \otimes F ) . \]
Hence $p := A{p'} + \frac{1}{2}\overline{B}$ is a symmetric principal part.

Let now $W_p$ be defined as in (\ref{W}). As mentioned above, the form $\omega$ in (\ref{standardsympform}) restricts to a regular symplectic form on $W_p$. A tedious but elementary calculation shows that
\[ (f', \phi') \ \mapsto \ \left( A(f') + \frac{B}{2}(\phi') , \phi' \right) \]
defines an isomorphism $\iota \colon W_{p'} \isom W_p$ satisfying $\iota^* \omega = \omega'$ and mapping $F \subset W_{p'}$ to $F \subset W_p$.

Part (b) is proven exactly as for extensions of line bundles in \cite[Lemma 6.6]{K}. \end{proof}

\subsection{Lagrangian subbundles in reference to a fixed symplectic extension}

From (\ref{W}), we obtain a splitting $\Rat (W) = \Rat (F) \oplus \Rat( F^* \otimes L )$. This is a vector space of dimension $\rank(W)$ over the field $K(C)$ of rational functions on $C$. If $\beta \in \Rat ( \Hom ( F^* \otimes L , F ))$, we write $\Gamma_{\beta}$ for the graph of the induced map of $K(C)$-vector spaces $\Rat (F^* \otimes L) \to \Rat (F)$. Abusing notation, we also denote by $\Gamma_\beta$ the associated sub-$\Oc$-module of $\sRat (F) \oplus \sRat (F^* \otimes L)$.

Moreover, if $F$ and $G$ are subsheaves of a sheaf $H$, we write $F \cap G$ for the presheaf $U \mapsto F(U) \cap G(U)$. As this is the kernel of the composed maps $F \to H \to H/G$ and $G \to H \to H/F$, in fact it is a sheaf.

\begin{prop} Let $p \in \Prin ( L^{-1} \otimes \Sym^2 F )$ be any symmetric principal part. Let $W_p$ be as in (\ref{W}). \label{cohomlifting}
\begin{enumerate} 
\item[(a)] There is a bijection between the $K(C)$-vector space $\Rat \left( L^{-1} \otimes \Sym^2 F \right)$ and the set of Lagrangian subbundles $E \subset W_p$ with $\rank (E \cap F) = 0$. The bijection is given by $\beta \mapsto \Gamma_{\beta} \cap W_p$. The inverse map sends a Lagrangian subbundle $E$ to the uniquely determined $\beta \in \Rat( L^{-1} \otimes \Sym^2 F)$ satisfying $\Rat(E) = \Gamma_{\beta}$.
\item[(b)] If $E = \Gamma_{\beta} \cap W_p$ then projection to $F^* \otimes L$ gives an isomorphism of sheaves $E \isom \Ker \left( (p - \overline{\beta}) \colon F^* \otimes L \to \sPrin (F) \right)$. Note that $\left[ p - \overline{\beta} \right] = [p] $ is the class of the symplectic extension  $\delta (W_p) \in H^1(C, L^{-1} \otimes \Sym^2 F)$.
\item[(c)] For a fixed $p - \overline{\beta} \in \Prin ( L^{-1} \otimes \Sym^2 F )$, the set of Lagrangian subbundles $\Gamma_{\beta'} \cap W_p$ with $\overline{\beta'} = \overline{\beta}$ is a torsor over $H^0 ( C, L^{-1} \otimes \Sym^2 F)$. In particular, it is nonempty.
\end{enumerate}
\end{prop}

\begin{proof} Parts (a) and (b) follow from \cite[Theorem 3.3 (i) and (iii)]{Hit1}. Note that as the symplectic form on $W$ is given by (\ref{standardsympform}), the $\alpha$ referred to in \cite{Hit1} is zero.

Part (c) is a slight generalization of \cite[Corollary 3.5]{Hit1}. From the description (\ref{W}), we see that $(\beta(\phi), \phi) \in \Gamma_\beta (U)$ belongs to $W_p (U)$ if and only if $\phi \in \Ker ( p - \overline{\beta} )(U)$, so $\Gamma_{\beta} \cap W_p$ is a lifting of $\Ker (p - \overline{\beta})$. By part (a), it is isotropic and saturated. 

Moreover, under the bijection in (a) the set of liftings $\Gamma_{\beta'} \cap W_p$ with $\overline{\beta'} = \overline{\beta}$ is in canonical bijection with the set of $\beta'$ such that $\overline{\beta'} = \overline{\beta}$. By (\ref{cohomseq}), this is a torsor over $H^0 ( C, L^{-1} \otimes \Sym^2 F )$. \end{proof}

\begin{remark} In part (c) above, we characterize different liftings of $\Ker(q)$ for a fixed $q \in \Prin(L^{-1} \otimes \Sym^2 F)$ with $\delta(W) = [q]$. More generally, there can exist also distinct $\beta, \beta'$ such that $\Ker (p - \overline{\beta}) = \Ker (p - \overline{\beta'})$ as subsheaves of $F^* \otimes L$. Such $\beta$ and $\beta'$ correspond to distinct liftings $E \hookrightarrow W$. We shall study this phenomenon in Lemma \ref{DifferentLiftings}. 
\end{remark}

We now give a slight refinement of Lemma \ref{sympext}, essentially allowing us to choose convenient coordinates on $W$.

\begin{lemma} Let $F$ and $E$ be Lagrangian subbundles of $W$ such that $\rank (F \cap E) =  0$. Then there exists a symmetric principal part $p_0 \in \Prin ( L^{-1} \otimes \Sym^2 F )$ and an isomorphism of symplectic bundles $\iota \colon W \isom W_{p_0}$, such that
\[ \iota(E) \ = \ \Gamma_0 \cap  W_{p_0} \ = \ 0 \oplus \Ker (p_0) , \]
where $\Gamma_0 = 0 \oplus \Rat (F^* \otimes L)$ is the graph of the zero map $\sRat (F^* \otimes L) \to \sRat ( F )$. \label{GoodCoords} \end{lemma}

\begin{proof} From Lemma \ref{sympext} and Proposition \ref{cohomlifting}, we may assume that $W$ is an extension
\[ 0 \ \to \ F \ \to \ W_p \ \to \ F^* \otimes L \ \to \ 0 \]
for a symmetric $p \in \Prin ( L^{-1} \otimes F \otimes F )$, and that
\[ E \ = \ \Gamma_\beta \cap W_p \ \cong \ \Ker ( p - \overline{\beta} ) \]
for some $\beta \in \Rat(L^{-1} \otimes \Sym^2 F)$. Then $(f, \phi) \mapsto (f - \beta(\phi), \phi)$ defines an isomorphism $\iota \colon W_p \isom W_{p-\overline{\beta}}$ sending $E = \Gamma_\beta \cap W_p$ to $\Gamma_0 \cap W_{p - \overline{\beta}}$. Set $p_0 := p - \overline{\beta}$. If $\omega$ and $\omega_0$ are the standard symplectic forms (\ref{standardsympform}) on $W_p$ and $W_{p_0}$ respectively, then an easy computation using the symmetry of $\beta$ shows that $\iota^* \omega_0 = \omega$. \end{proof}

\begin{remark} Apropos Lemma \ref{GoodCoords} and (\ref{W}): As $\Ker ( p_0 )$ is only a subsheaf of $F^* \otimes L$, it may be of interest to indicate how it lifts to a saturated subsheaf, or a subbundle of $W_{p_0}$. For simplicity, suppose $L = \Oc$ and $\Image ( p_0 ) \cong \cc_x$, so $p_0$ is represented by $\frac{\eta_1 \otimes \eta_1}{z}$ where $z$ is a local parameter at $x$ on a neighborhood $U$ and $\eta_1$ is some regular section of $F|_U$ which is nonzero at $x$.

Complete $\eta_1$ to a frame $\{ \eta_i \}$ for $F$ on $U$ and let $\{ \phi_i \}$ be the dual frame for $F^*$. Then the principal part $p_0 ( \phi_1 ) \in \sPrin (F)$ is represented by
\[ \frac{\eta_1 \otimes \eta_1}{z} ( \phi_1 ) \ = \ \frac{ \langle \eta_1 , \phi_1 \rangle \cdot \eta_1}{z} \ = \ \frac{\eta_1}{z} . \]
Hence in view of (\ref{W}), a frame for $W_p$ on $U$ is given by
\begin{equation} (\eta_1, 0) , \ldots , (\eta_n, 0 ) , \left( \frac{\eta_1}{z}, \phi_1 \right) ,  ( 0, \phi_2 ) , \ldots , ( 0 , \phi_n ) . \label{coordsW} \end{equation}
Now a frame over $U$ for the subsheaf $0 \oplus \Ker (p_0)$ of $W_{p_0}$ is given by
\begin{equation} (0, z \cdot \phi_1), (0, \phi_2 ) , \ldots , (0, \phi_n ). \label{frameKerp} \end{equation}
Writing $(0, z \cdot \phi_1)$ in terms of the frame (\ref{coordsW}), we have
\[ ( 0, z \cdot \phi_1 ) \ = \ z \cdot \left( \frac{\eta_1}{z} , \phi_1 \right) - ( \eta_1  , 0 ) . \]
From this we see that the images of (\ref{frameKerp}) in $W_p|_x$ are independent. Hence $0 \oplus \Ker (p_0) \hookrightarrow W_{p_0}$ is a vector bundle inclusion at $x$. This computation also shows that the intersection of the subbundles $\Gamma_0 \cap W_{p_0}$ and $F$ at $x$ is the line spanned by $\eta_1(x)$ in $F|_x$. \qed 
\end{remark}

\subsection{Isotropic liftings of an elementary transformation}

Let $W$ be a symplectic extension $0 \to F \to W_p \to F^* \otimes L \to 0$, and let $0 \to E \xrightarrow{\gamma} F^* \otimes L \to \tau \to 0$ be an elementary transformation where $\tau$ is some torsion sheaf. Assume that there is a lifting $j \colon E \to W$. By Proposition \ref{cohomlifting}, there exists a rational map $\beta \colon \sRat (F^* \otimes L) \to \sRat(F)$ such that $E \subseteq \Gamma_\beta \cap W_p \cong \Ker ( p - \overline{\beta} )$. The  following result, generalizing Proposition \ref{cohomlifting} (c), provides the main idea to ``linearize'' the space of Lagrangian subsheaves of $W$ which respects the fixed symplectic extension and elementary transformation.

\begin{lemma} \label{DifferentLiftings} The set of liftings of $\gamma \colon E \to F^* \otimes L$ to Lagrangian subsheaves of $W = W_p$ is a torsor over $H^0 \left( C, \Hom ( E, F ) \cap \sRat ( L^{-1} \otimes \Sym^2 F ) \right)$. \end{lemma}

\noindent Before starting the proof, let us indicate how the intersection of $\Hom (E, F)$ and $\sRat ( L^{-1} \otimes \Sym^2 F )$ is well defined. Since $L^{-1} \otimes F \xrightarrow{^t\gamma} E^*$ is an elementary transformation, $E^*$ is a subsheaf of $\sRat (L^{-1} \otimes F)$. Hence $\Hom (E, F) = E^* \otimes F$ and $\sRat ( L^{-1} \otimes \Sym^2 F)$ are both sub-$\Oc$-modules of $\sRat ( L^{-1} \otimes F \otimes F)$.

\begin{proof}
Suppose that $j_1 \colon E \to W$ and $j_2 \colon E \to W$ are two liftings of $\gamma$ to Lagrangian subsheaves. Then the saturations $\overline{j_i (E)}$ are Lagrangian subbundles. By Proposition \ref{cohomlifting} (a), there exist uniquely defined $\beta_1$, $\beta_2 \in \Rat (L^{-1} \otimes \Sym^2 F)$ such that for $i = 1, 2$ the map $j_i \colon E \to W_p$ is given by
\[ v \ \mapsto \ (\beta_i (v), \gamma (v)) \ \in \ W_p \ \subset \ \sRat(F) \oplus (F^* \otimes L) . \]
Then we calculate
\[ j_1 (v) - j_2 (v) \ = \ ( \beta_1(v), \gamma (v)) - (\beta_2(v), \gamma(v) ) \ = \ ( (\beta_1 - \beta_2)(v), 0 ) . \] 
Hence $j_1 - j_2$ defines an element of $H^0 \left( C, \Hom (E, F) \cap \sRat(L^{-1} \otimes \Sym^2 F) \right)$.

Conversely, suppose $v \mapsto (\beta(v), \gamma(v))$ is a lifting of $\gamma$ as above. If $\alpha \in \Rat ( L^{-1} \otimes \Sym^2 F )$ is regular on $\gamma(E) \subset F$, then $v \mapsto (\beta(v) + \alpha(v) , \gamma(v) )$ uniquely determines another rank $n$ subsheaf of $W_p$ lifting $\gamma(E)$. Since $\beta + \alpha$ is symmetric, by Proposition \ref{cohomlifting} (a), this subsheaf is isotropic. \end{proof}

Motivated by Lemma \ref{DifferentLiftings}, we make a definition.

\begin{definition} Let $0 \to E \xrightarrow{\gamma} F^* \otimes L \to \tau \to 0$ be as above. From $L^{-1} \otimes F \xrightarrow{^t\gamma} E^*$ we deduce an inclusion $L^{-1} \otimes F \otimes F \to E^* \otimes F$. We define $S_\gamma$ to be the saturation of $L^{-1} \otimes \Sym^2 F$ in $E^* \otimes F$. \label{defSgamma}
\end{definition}

\noindent Note that $S_\gamma$ depends only on $\gamma$, not on an extension $0 \to F \to W \to F^* \otimes L \to 0$.

\begin{lemma} \quad \label{Sgamma}
\begin{enumerate}
\item[(a)] There is an exact sequence $0 \to L^{-1} \otimes \Sym^2 F \to S_\gamma \to \tau_1 \to 0$, where $\tau_1$ is a torsion sheaf. In particular, $S_\gamma$ is locally free of rank $\frac{1}{2}n(n+1)$ and degree $\deg ( L^{-1} \otimes \Sym^2 F ) + \deg( \tau_1 )$.
\item[(b)] There is an exact sequence $0 \to S_\gamma \to L \otimes \Sym^2 E^* \to \tau_2 \to 0$, where $\tau_2$ is a torsion sheaf.
\item[(c)] If $\tau$ has reduced support, then $\tau_1$ is isomorphic to $\tau$. In particular, in this case $\deg (S_\gamma) = \deg (L^{-1} \otimes \Sym^2 F) + \deg (\tau)$. \end{enumerate}
\end{lemma}

\begin{proof}
(a) This follows from the definition of $S_\gamma$.

(b) Consider the following diagram (which is not exact), where the horizontal arrows are induced by ${^t\gamma}$.
\[ \xymatrix{ L^{-1} \otimes \Sym^2 F \ar@{^{(}->}[d] \ar[r] & S_\gamma \ar[r] \ar@{^{(}->}[d] & ( \Sym^2 E^* ) \otimes L \ar@{^{(}->}[d] \\
L^{-1} \otimes F \otimes F \ar[r] & E^* \otimes F \ar[r] & E^* \otimes E^* \otimes L } \]
As the horizontal arrows are isomorphisms at the generic point, we obtain (b).


(c) If the support of $\tau$ is reduced, then so is that of the torsion sheaf
\[ \frac{ E^* }{{^t\gamma}(L^{-1} \otimes F)} . \]
At each $x \in \Supp (\tau)$, a local basis for $E^* \subset \sRat ( L^{-1} \otimes F )$ is given by
\[ \frac{\lambda \otimes \eta_1}{z} , \lambda \otimes \eta_2, \cdots, \lambda \otimes \eta_n , \]
where $\{ \eta_1, \ldots , \eta_n \}$ is a suitable local basis of $F$ and $\lambda$ a local generator of $L^{-1}$, and $z$ is a local parameter at $x$. Then a local basis of $E^* \otimes F$ is given by
\[
\left\{ \frac{\lambda \otimes \eta_1 \otimes \eta_k}{z} : 1 \le k \le n \right\} \: \cup \: \left\{\lambda \otimes \eta_m \otimes \eta_k : \begin{array}{c} 2 \le m \le n ; \\ 1 \le k \le n \end{array} \right\}.
\]
Thus a local basis of $S_\gamma$ is given by
\[ \left\{ \frac{\lambda \otimes \eta_1 \otimes \eta_1}{z} \right\} \ \cup \ \left\{ \frac{1}{2} \left( \lambda \otimes  \eta_k \otimes \eta_m + \lambda \otimes \eta_m \otimes \eta_k \right) : \begin{array}{l} 1 \le k,m \le n ; \\ (m, k) \ne (1, 1) \end{array} \right\} . \]
Therefore, in this case $\tau_1$ is a sum of torsion sheaves of degree 1, each supported at one of the points $x \in \Supp (\tau)$. The statement follows. \end{proof}

\begin{remark} \label{NonReduced} Lemma \ref{Sgamma} (c) is false if $\tau_1$ does not have reduced support. For example, suppose that $L = \Oc$ and $\tau = \cO_x \oplus \cO_x$, so $E^*$ is spanned near $x$ by $\frac{\eta_1}{z} , \frac{\eta_2}{z}, \eta_3 , \ldots , \eta_n$ for suitable $\eta_i$. Then $S_\gamma$ is spanned near $x$ by
\[ \frac{\eta_1 \otimes \eta_1}{z} , \quad \frac{\eta_1 \otimes \eta_2 + \eta_2 \otimes \eta_2}{z} , \quad \frac{\eta_2 \otimes \eta_2}{z}, \quad \begin{array}{l} \frac{1}{2} ( \eta_i \otimes \eta_j + \eta_j \otimes \eta_i ) : 1 \le i \le j \le n; \\ 
 (i, j) \not\in \{ (1, 1), (1, 2), (2, 2) \}. \end{array} \]
Thus $\tau_1 = \cO_x^{\oplus 3} \not\cong \tau$.
\end{remark}

\subsection{Geometry in extension spaces} \label{GeomCrit}

Let $F \to C$ be a bundle of rank $n$, and consider the scroll $\pi \colon \pp F \to C$. Throughout this subsection, we shall assume that $h^1 ( C, L^{-1} \otimes \Sym^2 F ) \ne 0$.

By Serre duality and the projection formula, there is an isomorphism
\[ \pp H^1 ( C, L^{-1} \otimes \Sym^2 F) \ \xrightarrow{\sim} | \cO_{\pp F} (2) \otimes \pi^* (\Kc L) |^* . \]
Thus we obtain a natural map $\psi \colon \pp F \dashrightarrow \pp H^1 (C, L^{-1} \otimes \Sym^2 F )$ with nondegenerate image.

We shall use an explicit description of $\psi$, given in \cite[{\S} 2]{CH2}. For each $x \in C$, there is a sheaf sequence
\begin{equation} 0 \to L^{-1} \otimes \Sym^2 F \to L^{-1}(x) \otimes \Sym^2 F \to \frac{L^{-1}(x) \otimes \Sym^2 F}{L^{-1} \otimes \Sym^2 F} \to 0 . \label{fiberwise} \end{equation}
Taking global sections, the associated long exact sequence is a subsequence of (\ref{cohomseq}) for $V = L^{-1} \otimes \Sym^2 F$. The following is easy to check by explicit computation.

\begin{lemma} The map $\psi$ can be identified fiberwise with the projectivization of the coboundary map in the associated long exact sequence of (\ref{fiberwise}), restricted to the image of the Segre embedding $\pp F|_x \hookrightarrow \pp ( L^{-1} \otimes \Sym^2 F )|_x$. In particular, the image of $\eta \in \pp F|_x$ is defined by the cohomology class of a principal part of the form $\frac{\lambda \otimes \eta \otimes \eta}{z}$, where $z$ is a local parameter at $x$ and $\lambda$ a local generator of $L^{-1}$. \label{alternativePsi} \end{lemma}

\begin{remark} Although we do not use this fact, we mention that $\psi$ is an embedding if $F$ is stable and $\deg ( F ) < n(\frac{\ell}{2} - 1)$ (see \cite[Lemma 2.6]{CH2} for the case where $L = \Oc$). The important property of $\psi$ for us will be that the image is nondegenerate. This is central to Proposition \ref{QFcDense}.
\end{remark}

Now let $\eta_1 , \ldots , \eta_t$ be points of $\pp F|_{x_1} , \ldots , F|_{x_t}$ for distinct $x_1 , \ldots , x_t \in C$, and $\gamma \colon E \to F^* \otimes L$ be the corresponding elementary transformation. Let $S_\gamma$ be as in Definition \ref{defSgamma}.

\begin{lemma} \label{vanishinglemma} We have $h^1 ( C, S_\gamma ) = 0$ if and only if the points $\psi (\eta_1) , \ldots , \psi (\eta_t)$ span $\pp H^1 (C, L^{-1} \otimes \Sym^2 F)$. \end{lemma}

\begin{proof} The proof of Lemma \ref{Sgamma} (c) shows that $S_\gamma$ is an elementary transformation
\[
0 \longrightarrow L^{-1} \otimes \Sym^2 F \longrightarrow S_\gamma \longrightarrow \bigoplus_{k=1}^t \cc \cdot \frac{\lambda_k \otimes \eta_k \otimes \eta_k}{z_k} \longrightarrow 0 ,
\]
where the $z_k$ and $\lambda_k$ are defined analogously as in Lemma \ref{Sgamma} (c). In view of Lemma \ref{alternativePsi}, the lemma follows from the associated long exact sequence is
\begin{multline} \label{SgammaExactSeq}
0 \ \longrightarrow \ H^0 ( C, L^{-1} \otimes \Sym^2 F ) \ \longrightarrow \ H^0 ( C, S_\gamma ) \ \longrightarrow \ \cc^t \\ \longrightarrow \ H^1 ( C, L^{-1} \otimes \Sym^2 F) \ \longrightarrow \ H^1( C, S_\gamma) \ \longrightarrow \ 0 . \qedhere \end{multline}
\end{proof}

\begin{remark} \label{syzygies} Suppose that $h^0 (C, L^{-1} \otimes \Sym^2 F) = 0$. By exactness of (\ref{SgammaExactSeq}), we see that $H^0 (C, S_\gamma)$ is the vector space of linear relations of the points $\psi ( \eta_k )$ in $\pp H^1 (C, L^{-1} \otimes \Sym^2 F )$. \end{remark}

\section{Irreducibility of Lagrangian Quot schemes} \label{irred}

Let $W$ be an $L$-valued symplectic bundle of rank $2n$, where $\deg L = \ell$. In general, the Lagrangian Quot schemes $\LQ$ can be reducible, and also there may be irreducible components whose points all correspond to non-saturated subsheaves. In this section, we shall prove the following theorem, showing that for sufficiently large $e$, these phenomena disappear.

\begin{theorem} \label{thmA} Let $W$ be an arbitrary $L$-valued symplectic bundle over $C$. Then there exists an integer $e (W)$ such that for $e \ge e(W)$, the Lagrangian Quot scheme $\LQ$ is irreducible and generically smooth of dimension $(n+1)e + \frac{1}{2}n(n+1)(\ell -g+1)$, and a general point of $\LQ$ corresponds to a Lagrangian subbundle. Moreover, when $g \ge 2$, a sufficiently general point of $\LQ$ defines a stable vector bundle. \end{theorem}

\begin{remark} Recall that $\LQo$ denotes the open sublocus of $\LQ$ corresponding to vector bundle quotients. Theorem \ref{thmA} shows in particular, for large $e$, that $\LQ$ is a compactification of $\LQo$ which is in fact the closure of $\LQo$ in $\Quot_{n, -e} (W)$. Other compactifications of $\LQo$ have also been studied; more generally, generalizations of Quot schemes to principal $G$-bundles: Hilbert schemes of sections of $\LG(W)$ as in \cite{HN} and moduli of stable maps to $\LG(W)$ as in \cite{H} and \cite{PR}. One attractive feature of $\LQ$ is that it naturally supports a universal family of sheaves, inherited from $\Quot_{n, -e} (W)$. \end{remark}

Before embarking on the proof of Theorem \ref{thmA}, let us compare our approach with the proof of the analogous statement for $\Quot_{n, -e} (W)$ in \cite{PR}. Replacing the Grassmannian bundle $\Gr (n, V)$ with the Lagrangian Grassmannian bundle $\LG (W)$, the argument of \cite[{\S} 3]{PR} shows the dimension bound
\[ \dim \LQ \ \le \ \binom{n+1}{2} + (n+1)(e - e_0) \]
where $-e_0$ is the degree of a maximal Lagrangian subbundle of $W$. However, the following difficulty arises.

If $V$ and $E$ are vector bundles of rank $N$ and $n$ respectively with $n < N$, then sheaf injections $E \to V$ are parameterized by an open subset of the linear space $H^0 (C, \Hom(E, V))$. One can then construct the irreducible space of stable rank $n$ subsheaves of $V$ as in \cite[Proposition 6.1]{PR}. However, for a symplectic bundle $W$, isotropic subsheaves $[j \colon E \to W]$ define a locally closed subset of $H^0(C, \Hom(E, W))$.  This seems to be a nonlinear subvariety, for which there is no guarantee of irreducibility.

To overcome this difficulty, we introduce a collection of auxiliary Lagrangian subbundles $F$ of $W$. It will emerge in view of Lemma \ref{DifferentLiftings} that Lagrangian subsheaves can be parameterized in a linear way if one also records how they are related to a fixed such $F$.

We proceed to the first ingredient of the proof, which is a result on the evaluation maps $\ev_x^e \colon \LQo \to \LG ( W|_x )$.

\subsection{Surjectivity of evaluation maps}

Recall that if $\cF \to B \times C$ is a family of objects over $C$, we denote by $\cF_b$ the restriction $\cF|_{\{ b \} \times C}$.

\begin{lemma} Let $\cW \to B \times C$ be a family of bundles of rank $r$ parameterized by an irreducible base $B$. Then there exists an integer $m_0 ( \cW )$ such that for $d \ge m_0 ( \cW )$, the evaluation map $\ev_x^d \colon \Quot_{1, -d} ( \cW_b )^\circ \to \pp \cW_b|_x$ is surjective. \label{evSurjRankOne} \end{lemma}

\begin{proof} Fix $p \in C$, and let $\Oc (p)$ be the corresponding effective line bundle of degree $1$ on $C$. By a semicontinuity argument, there is an integer $m_0 = m_0 ( \cW )$ such that for all $d \ge m_0$ and for all $b \in B$, the evaluation map on sections
\[
\phi_d \colon H^0 ( C , \cW_b (dp) ) \otimes \Oc \ \longrightarrow \ \cW_b (dp)
\]
is a surjective bundle map. 
 Thus the evaluation map $\Quot_{1, -d} ( \cW_b ) \dashrightarrow \cW_b|_x$ is surjective for all $d$ and $(b, x)$. We now show that its restriction to the locus $\Quot_{1, -d} ( \cW_b )^\circ$ of saturated subsheaves is surjective.

Choose a generating subspace $V \subseteq H^0 ( C, \cW_b (dp) )$ of dimension $r + 1$. Let $\psi \colon C \to \pp V$ be the map sending $x$ to the point defined by the one-dimensional subspace $V \cap H^0 ( C, \cW_b (-x) )$. It is not hard to see that $\psi$ is induced by a subsystem of $|\det ( \cW_b )|$, so $\psi (C)$ is nondegenerate.

Now let $\lambda$ be any line in $\cW_b|_x$. Write $V_\lambda := \phi_d^{-1} ( \lambda )$. As $\pp V_\lambda$ is a line in $\pp V$ and $\psi (C)$ is nondegenerate, clearly $\pp V_\lambda \cap \psi (C)$ is a finite set (containing $\psi(x)$). Any $s \in V_\lambda$ not lying over this finite set is a nowhere vanishing section of $\cW_b ( dp )$ which spans $\lambda$ at $x$. 
 The lemma follows. \end{proof}

\begin{prop} \label{evSurj} Let $\cW \to B \times C$ be a family of $L$-valued symplectic bundles of rank $2n$ parameterized by an irreducible base $B$. Then there exists an integer $f_0 ( \cW )$ such that if $f \ge f_0 ( \cW )$, then for all $(b, x) \in B \times C$, the evaluation map $LQ_{-f} ( \cW_b )^\circ \to \LG ( \cW_b|_x )$ is surjective. \end{prop}

\begin{proof} We shall prove the lemma by induction on $n$. If $\rank ( \cW ) = 2$ then, as any line subbundle is isotropic, it suffices to set $f_0 ( \cW ) = m_0 ( \cW )$ as in Lemma \ref{evSurjRankOne}.

Now suppose $2n \ge 4$. Set $m_0 = m_0 ( \cW )$ as in Lemma \ref{evSurjRankOne} and consider the relative Quot scheme and universal line bundle
\[ \cP \ \to \ \cQ uot_{1, -m_0} ( \cW )^\circ \times C \ \to \ B \times C \]
parameterizing degree $-m_0$ line subbundles of all the $\cW_b$. As any line subbundle is isotropic, $\cP^\perp / \cP$ is a family of $L$-valued symplectic bundles of rank $2n-2$ parameterized by the total space of $\cQ uot_{1, -m_0} ( \cW )^\circ \to B$. Since $\cQ uot_{1, -m_0} ( \cW )^\circ$ is quasi-projective over $B$, it has finitely many irreducible components. By induction, we may assume there exists an integer $f_0 ( \cP^\perp / \cP )$ such that for any $b \in B$ and any $P \in \Quot_{1, -m_0} ( \cW_b )^\circ$, if $a \ge f_0 ( \cP^\perp / \cP )$ then the evaluation map
\[ LQ_{-a} ( P^\perp / P )^\circ \ \to \ \LG \left( P^\perp / P \right)|_x \]
is surjective for all $x \in C$.

Now we return to the original family $\cW \to B \times C$. For any $(b, x)$, let $\Lambda$ be a Lagrangian subspace of a fiber $\cW_b|_x$. Choose any line $\lambda \subset \Lambda$. By Lemma \ref{evSurjRankOne}, we may choose a line subbundle $P \subset \cW_b$ of degree $-m_0$ with $P|_x = \lambda$. Then $\Lambda / \lambda$ is a Lagrangian subspace of $(P^\perp / P)|_x$. By the previous paragraph, for any $a \ge f_0 ( \cP^\perp / \cP )$ we may assume there exists $\tE \in LQ_{-a} ( P^\perp / P )^\circ$ such that $\tE|_x = \Lambda / \lambda$. The inverse image of $\tE$ in $\cW_b$ is a Lagrangian subbundle $E$ of $\cW_b$ of degree $\deg ( \tE ) + \deg (P) = -a - m_0$ satisfying $E|_x = \Lambda$. 
 Setting $f_0 ( \cW ) = m_0 + f_0 ( \cP^\perp / \cP )$, we have proven the proposition. \end{proof}

\subsection{Proof of Theorem \ref{thmA}}

Fix now an arbitrary $L$-valued symplectic bundle $W$ of rank $2n \ge 2$. We write $f_0$ for $f_0 (W)$ as defined in Proposition \ref{evSurj}, where $W$ is regarded as a family with one element. We now introduce the ``auxiliary'' Lagrangian subbundles $F$ mentioned at the start of {\S} \ref{irred}.

\begin{definition} Let $F$ be a Lagrangian subbundle of $W$, and write $\pi \colon W \to F^* \otimes L$ for the quotient map. We define
\[ Q^e_{F, \pi} \ := \ \{ E \in \LQ : \rank ( E \cap F ) \ = \ 0 \} . \]
When the surjection $\pi \colon W \to F^* \otimes L$ is clear from the context, we denote $Q^e_{F, \pi}$ simply by $\qf$ to ease notation. \end{definition}

\begin{remark}
\begin{enumerate}
\item[(a)] In view of the exact sequence $0 \to F \to W \xrightarrow{\pi} F^* \otimes L \to 0$, if $E \in \qf$, then $E$ is an elementary transformation of $F^* \otimes L$. Therefore, $\qf$ is nonempty only if $e \ge \deg (F) - n\ell$.
\item[(b)] For any Lagrangian subsheaf $F \subset W$ and any $e \ge f_0 (W)$, by Proposition \ref{evSurj} we can find $[j \colon E \to W] \in \LQo$ such that $E|_x \cap F|_x = 0$ for some and hence for general $x \in C$. Thus for any Lagrangian subbundle $F$ and any $e \ge f_0 (W)$, the locus $\qf$ contains a component whose general
member is saturated.
\item[(c)] Clearly $\qf$ is open in all components of $\LQ$, although it may be empty in some.
\end{enumerate}
\end{remark}

\noindent \textbf{In what follows, we shall always assume the ``auxiliary'' bundles $F$ have degree $-f_0 (W) =: -f_0$}.
This will give the best bound $e(W)$ in Theorem \ref{thmA} available with these methods.

\begin{prop} \quad \begin{enumerate} \item[(a)] Any Lagrangian subsheaf $E \subset W$ belongs to $Q_F^e$ for some Lagrangian subbundle $F$ of degree $-f_0$, where $\deg (E) = -e$. In particular, for any $e$, as $F$ varies in $LQ_{-f_0} (W)^\circ$ the loci $\qf$ form an open covering of $\LQ$.
\item[(b)] Suppose now that $e \ge f_0$. Then for $F, F' \in LQ_{-f_0} ( W )^\circ$, the intersection $\qf \cap Q_{F'}^e$ is nonempty.
\end{enumerate}
\label{cover} 
\end{prop}

\begin{proof} (a) Let $E$ be any Lagrangian subsheaf of $W$. By Proposition \ref{evSurj}, we can find a Lagrangian subbundle $F$ of degree $-f_0$ intersecting $E|_x$ in zero at some $x \in C$. Thus $[E \to W]$ belongs to $\qf$, where $\deg (E) = -e$.

(b) We must find a Lagrangian subsheaf $E$ of degree $-e$ intersecting both $F$ and $F'$ generically in rank zero. For some $x \in C$, choose $\Lambda \in \LG ( W|_x )$ intersecting both $F|_x$ and $F'|_x$ in zero. As by hypothesis $e \ge f_0$, by Proposition \ref{evSurj} we can find a Lagrangian subbundle $E$ of degree $-e$ satisfying $E|_x = \Lambda$. Then $[E \to W]$ is a point of $\qf \cap Q_{F'}^e$. \end{proof}

Next, for any bundle $G$, we denote by $\Elm^t ( G )$ the Quot scheme $\Quot^{0, t} ( G )$ parameterizing torsion quotients of degree $t$; that is, elementary transformations $G' \subset G$ with $\deg (G / G') = t$.

Now let $F$ be any degree $-f_0$ Lagrangian subbundle of $W$. For any $e$, given an element $[j \colon E \to W]$ of $\qf$, by composing with $\pi \colon W \to F^* \otimes L$ we get an elementary transformation $\pi \circ j \colon E \to F^* \otimes L$. The association $j \mapsto \pi \circ j =: \wj$ defines a morphism
\[ \pi_* \colon \qf \ \to \ \Elm^{e + f_0 + n\ell} (F^* \otimes L) . \]
We now study a certain subset of $\qf$ with some desirable properties. To ease notation, we set $t = t(e) := e + f_0 + n\ell$.

\begin{definition} \label{DefnQFc} For each $F$ as above, let $\qfc$ be the subset of $\qf$ of subsheaves $[j \colon E \to W]$ such that
\begin{enumerate}
\item[(i)] $E$ is saturated in $W$; that is, $j$ is a vector bundle injection;
\item[(ii)] $(F^* \otimes L) / \wj(E) \in \Elm^t ( F^* \otimes L)$ has reduced support; and
\item[(iii)] $h^1 (C, S_{\wj} ) = 0$.
\end{enumerate}
\end{definition}

\begin{remark}
\begin{enumerate}
\item[(a)] Note that conditions (ii) and (iii) depend only on the map $\wj \colon E \to F^* \otimes L$, and not a priori on $W$.
\item[(b)] By Remark \ref{NonReduced}, the family of sheaves over $\Elm^t ( F^* \otimes L )$ with fiber $S_\wj$ at $\wj$ is only flat over the locus of $\wj$ where $(F^* \otimes L) / \wj(E)$ has reduced support. Therefore, (iii) itself does not define an open condition on $\qf$, but (ii) and (iii) together define an open condition. Thus $\qfc$ is open in $\qf$.
%
\item[(c)] If $h^1 (C, L^{-1} \otimes \Sym^2 F ) = 0$ then (iii) follows from Lemma \ref{Sgamma} (a). Otherwise, by Lemma \ref{vanishinglemma}, if $(F^* \otimes L) / \wj (E)$ has reduced support then (iii) is equivalent to the points $\eta_1 , \ldots , \eta_t \in \pp F$ corresponding to the elementary transformation $E \subset F^* \otimes L$ spanning $\pp H^1 ( C, L^{-1} \otimes \Sym^2 F)$.
\end{enumerate} \end{remark}

\begin{remark} In view of conditions (ii) and (iii) and (\ref{SgammaExactSeq}), the locus $\qfc$ is nonempty only if $f_0 + n\ell + e \ge h^1 ( C, L^{-1} \otimes \Sym^2 F)$. By Riemann--Roch, this becomes
\[ e \ \ge \ n f_0 + \frac{n(n+1)}{2} ( g - 1 ) + \frac{n(n-1)}{2} \ell + h^0 ( C, L^{-1} \otimes \Sym^2 F ) . \]
Now for any $F \subset W$ we have $h^0 ( C, L^{-1} \otimes \Sym^2 F ) \le h^0 ( C, L^{-1} \otimes \Sym^2 W )$. In order to obtain later a bound which will apply to $\qfc$ for all $F$, \textbf{in what follows, we shall always assume that}
\begin{equation} e \ \ge \ e_1 (W) \ := \ n f_0 + \frac{n(n+1)}{2} ( g - 1 ) + \frac{n(n-1)}{2} \ell + h^0 ( C, L^{-1} \otimes \Sym^2 W ) + 1 . \label{eoneW} \end{equation}
(The final $+1$ term is required for technical reasons in Proposition \ref{QFcDense}.)
\end{remark}

\begin{prop} \label{vanishing} For any $[j \colon E \to W] \in \qfc$, the following holds.
\begin{enumerate}
\item[(a)] We have $h^1 ( C, L \otimes \Sym^2 E^* ) = 0$.
\item[(b)] The locus $\qfc$ is smooth and of the expected dimension $\chi ( C, L \otimes \Sym^2 E^* )$ at $E$.
\end{enumerate} \end{prop}

\begin{proof} By definition of $\qfc$, the subsheaf $j(E)$ is saturated in $W$. From Lemma \ref{Sgamma} (b) it follows that $H^1 (C, L \otimes \Sym^2 E^*)$ is a quotient of $H^1 ( C, S_{\wj} )$. As the latter is zero by definition of $\qfc$, we obtain statement (a). Part (b) now follows from Proposition \ref{Zariski} (c). 
\end{proof}

\begin{prop} \label{QFcDominates} Let $X$ be a nonempty irreducible component of $\qfc$. 
 Then for $t = e + f_0 + n \ell$, the map $\pi_* \colon X \to \Elm^t (F^* \otimes L)$ is dominant and has irreducible fibers. \end{prop}

\begin{proof} For any $[j \colon E \to W] \in X$, by Proposition \ref{vanishing}, we have
\[ \dim ( X ) \ = \ \chi ( C, L \otimes \Sym^2 E^* ) . \]
Moreover, by Lemma \ref{DifferentLiftings}, the fiber $\pi_*^{-1} \left( \wj \right)$ is an open subset of a torsor over $H^0 ( C, S_{\wj} )$. Hence it is irreducible, and of dimension $h^0 ( C, S_\wj )$. As $h^1 ( C, S_\wj ) = 0$ by definition of $\qfc$, in fact $\dim ( \pi_*^{-1} ( \wj ) ) = \chi ( C, S_\wj )$. Thus $\pi_* (X)$ has dimension at least
\[ \chi ( C, L \otimes \Sym^2 E^*) - \chi ( C, S_{\wj} ) \ = \ \deg ( L \otimes \Sym^2 E^* ) - \deg ( S_\wj ) , \]
the last equality using Lemma \ref{Sgamma} (b). Now by definition of $\qfc$, the torsion sheaf $(F^* \otimes L)/\wj(E)$ has reduced support. Using Lemma \ref{Sgamma} (c), we compute that
\[ \deg (L \otimes \Sym^2 E^*) - \deg(S_\wj) \ = \ nt , \]
which is exactly $\dim \Elm^{t} (F^* \otimes L)$. 
 Therefore, $\pi_* (X)$ is dense in $\Elm^t (F^* \otimes L)$ as the latter is irreducible.
 \end{proof}

\begin{prop} For any $F \in LQ_{-f_0} ( W )^\circ$, the locus $\qfc$ is irreducible. \label{QFcIrr} \end{prop}

\begin{proof} Suppose $X_1$ and $X_2$ were distinct irreducible components of $\qfc$. By Proposition \ref{QFcDominates}, the restriction of $\pi_*$ to either component is dominant with irreducible fibers. Therefore, $X_1$ and $X_2$ would have to intersect along a dense subset of a generic fiber. But this would contradict the smoothness of $\qfc$ proven in Proposition \ref{vanishing}. Thus $\qfc$ is irreducible. \end{proof}

The following key result shows the density of the well-behaved sublocus $\qfc \subset \qf$ for sufficiently large $e$.

\begin{prop} Let $e_1 ( W )$ be as defined in (\ref{eoneW}). For $e \ge e_1 ( W )$, the locus $\qfc$ is nonempty and dense in $\qf$. \label{QFcDense} \end{prop}

As the proof of this proposition is somewhat involved, we postpone it to {\S} \ref{ProofOfQFcDense}. The following is immediate from Propositions \ref{QFcIrr} and \ref{QFcDense}.

\begin{cor} For $e \ge e_1 (W)$, the locus $\qf$ is nonempty and irreducible. \label{QFIrr} \end{cor}

Now we can prove Theorem \ref{thmA}.

\begin{proof}[Proof of Theorem \ref{thmA}] By Proposition \ref{cover} (a), the loci $\qf$ are nonempty and cover $\LQ$. By Corollary \ref{QFIrr}, for $e \ge e_1 ( W )$, each $\qf$ is dense in exactly one component of $\LQ$. By Proposition \ref{cover} (b), if $e \ge f_0 (W)$, this must be the same component for all $F$. Therefore, $\LQ$ has only one irreducible component.

Furthermore, by Proposition \ref{QFcDense}, each $\qfc$ is in fact dense in $\LQ$. Hence, by Proposition \ref{vanishing}, a general point of $\LQ$ is smooth and represents a vector subbundle.

Finally, we show that general $E \in \LQ$ is stable as a vector bundle. For fixed $F \in LQ_{-f_0} ( W )^\circ$, if $t = e + f_0 + n\ell \ge n^2 (g-1) + 1$, then a general stable bundle $E$ of degree $-e$ occurs as an elementary transformation of $F^* \otimes L$. By Proposition \ref{QFcDominates}, if we assume that $e \ge \max\{ e_1 (W), n^2 (g-1) + 1 - f_0 - n \ell \}$ then a general element of $\Elm^t (F^* \otimes L)$ lifts to $W$. Hence, since $\LQ$ is irreducible, a general $E \in \LQ$ defines a stable vector bundle.

In summary, setting
\[ e(W) \ = \ \max\{ f_0 (W) , e_1 (W) , n^2 (g-1) + 1 - f_0 (W) - n \ell \} , \]
we obtain Theorem \ref{thmA}. \end{proof}

In analogy with \cite[Proposition 6.3]{PR}, Theorem \ref{thmA} implies immediately the following:

\begin{cor} If $g \ge 2$, then every symplectic bundle $W$ of rank $2n \ge 2$ can be fitted into a symplectic extension $0 \to E \to W \to E^* \otimes L \to 0$ where $E$ is a stable bundle. \end{cor}

\subsection{Proof of Proposition \ref{QFcDense}} \label{ProofOfQFcDense}

We shall prove Proposition \ref{QFcDense} by showing that for any $[ E \to W ] \in \qf \setminus \qfc$, there exists a one-parameter deformation of $E$ in $\qf$ of which a general member belongs to $\qfc$. We shall use principal parts to construct this deformation explicitly. We begin by discussing families of principal parts and extensions. It will be convenient to work in slightly greater generality than in {\S} \ref{SymmPptSympExt}.

\subsubsection{Families of principal parts and extensions} \label{FamiliesPptExt}

Let $V$ be any vector bundle. For any $d \ge 1$, the set of $V$-valued principal parts with poles bounded by a divisor of degree $d$ is naturally the total space of a vector bundle $\cT_d (V) \to \Cd$ with fiber $H^0 ( C, V(D) / V )$ at $D \in \Cd$.

Now suppose $V$ is a subbundle of $\Hom (F_2, F_1)$ for bundles $F_1$ and $F_2$ over $C$. Let $\pi_C \colon \cT_d (V) \times C \to C$ be the projection. There is a natural map
\[ P \colon \pi_C^* F_2 \ \to \ \pi_C^* \sPrin ( F_1 ) \]
of sheaves over $\cT_d (V) \times C$, given on stalks by sending
\[ f \ \in \ (\pi_C^* F_2)_{(p, x)} \quad \hbox{to} \quad P(f) = p_x ( f ) \ \in \ \pi_C^* \sPrin (F_1)_{(p, x)} . \]

Using this, we can globalize the construction (\ref{W}) to a ``Poincar\'e bundle'' over $\cT_d (V) \times C$. Let $\cW \subset \pi_C^* (\sRat (F_1) \oplus F_2)$ be the subsheaf given on open subsets $U \subseteq \cT_d (V) \times C$ by
\begin{equation} \cW ( U ) \ = \ \{ ( f, f' ) \in \pi_C^* (\sRat (F_1 \oplus F_2) (U) : p_x ( f' ) = \overline{f} \hbox{ for all } (p, x ) \in U \} . \label{relW} \end{equation}
Clearly $\cW|_{ \{ p \} \times C}$ generalizes the extension $W_p$ defined in (\ref{W}). In particular, as in Lemma \ref{sympext} (b), one shows that the cohomology class of the extension $\cW_{ \{ p \} \times C}$ satisfies
\begin{equation} \delta \left( \cW|_{ \{ p \} \times C} \right) = [ p ] . \label{familyCohomClass} \end{equation}

Write $\cE := \Ker ( P )$, the subsheaf of $\pi_C^* F_2$ given on $U$ as above by $p_x (f') = 0$ for each $(p, x) \in U$. This is a family of elementary transformations of $F_2$. The following lemma highlights an important aspect of such families of sheaves.

\begin{lemma} \label{FamilyLifting} The family $\cE$ lifts to a family of saturated subsheaves of $\cW$. \end{lemma}

\begin{proof} By construction of (\ref{relW}), the image of the inclusion map
\[ \gamma \colon \cE \ \to \ \ \pi_C^* F_2 \ \to \ \pi_C^* \left( \sRat (F_1 ) \oplus F_2 \right) \]
in fact belongs to $\cW$. For saturatedness: For each $p \in \cT_d (V)$, we have
\[ \gamma \left( \cE_p \right) \ = \ \Gamma_0 \cap \cW_p \ \cong \ \Ker ( p ) . \]
Hence, by Proposition \ref{cohomlifting}, in fact $\gamma ( \cE_p )$ is saturated in $\cW_p$. \end{proof}

\begin{remark} The family of sheaves $\cE \subset \pi_C^* F_2$ over $\cT_d (V) \times C$ is not flat. For example, if $d = 1$ then $p \in \cT_1 (V)$ determines a point of $\Hom ( F_2, F_1 (x) )|_x$ for some $x \in C$, and $\deg (\cE_p)$ depends on the rank of the corresponding map $F_2|_x \to F_1 (x)|_x$. Moreover, $\cT_{d'} (V) \subset \cT_d (V)$ for $1 \le d' \le d$, which may also lead to jumps in $\deg (\cE_p)$. 
\end{remark}

\begin{definition} A \textsl{family of $V$-valued principal parts} parameterized by a scheme $S$ is a map $p \colon S \to \cT_d (V)$ for some $d \ge 1$. \end{definition}

We shall most often denote such a family by $\{ p_s : s \in S \}$ or just $\{ p_s \}$ if no confusion should result.

For $V \subseteq \Hom (F_2, F_1)$, any family $p \colon S \to \cT_d (V)$ induces by pullback a family of extensions $\left\{ 0 \to F_1 \to \cW_{p_s} \to F_2 \to 0 \right\}$ parameterized by $S$, together with a family of subsheaves $\{ \cE_s  \subset \cW_{p_s} : s \in S \}$, where as above $\cE_s \cong \Ker ( p_s )$ for each $s$. By Lemma \ref{FamilyLifting}, every $\cE_s$ is saturated in $\cW_s$, although $\cE$ is in general not flat over $S$.

\subsubsection{General principal parts}

Let $V$ be a vector bundle. A principal part $p \in \Prin (V)$ will be called \textsl{general} if it can be represented by a sum
\begin{equation} \sum_{i=1}^m \frac{\sigma_i}{z_i} \label{GeneralPpt} \end{equation}
where $z_1 , \ldots , z_m$ are local parameters at distinct points $x_1 , \ldots , x_m$ of $C$ respectively, and $\sigma_i$ is a frame element for $V$ near $x_i$. If $h^1 ( C, V ) \ne 0$, then by a similar argument to that in Lemma \ref{alternativePsi}, the cohomology class $\left[ \frac{\sigma_i}{z_i} \right]$ defines the image of the point $\sigma_i (x_i)$ in $\psi ( \pp V ) \subseteq \pp H^1 ( C, V )$.

We recall that a finite set of points $x_1 , \ldots x_r \in \cc^{N + 1}$ (resp., $\pp^N$) is said to be \textsl{in general position} if for $1 \le k \le r$, the span of any $k$ of the $x_i$ has dimension $\min \{ k , N + 1 \}$ (resp., $\min \{ k - 1, N \}$).

In what follows, $Y$ will denote a nonempty closed subfibration of $\pp V \to C$ which is Zariski locally trivial, and let $\tY \subset V$ be the relative cone over $Y$, which is a Zariski locally trivial closed subfibration of $V$ invariant under fiberwise scalar multiplication. 

\begin{definition} \label{genYvaluedppt} We shall say that a general $p \in \Prin (V)$ as in (\ref{GeneralPpt}) is a \textsl{general $\tY$-valued principal part} if the following conditions are satisfied.

\begin{itemize}
\item Each $\sigma_i$ is a section of the subfibration $\tY \subset V$ near $x_i$. (Note that this is stronger than the property that $\sigma_i (x_i) \in \tY|_{x_i}$ for each $i$.) 
\item If $h^1 ( C, V ) \ne 0$, then the classes $\left[ \frac{\sigma_i}{z_i} \right]$ are in general position in $H^1 ( C, V )$.
\end{itemize}
\end{definition}

\noindent Note that we do not directly define ``$\tY$-valued principal parts'', but only ``general $\tY$-valued principal parts''. 

In the case of interest to us, $V = L^{-1} \otimes \Sym^2 F$ and $Y$ is the relative Segre embedding $\pp F \hookrightarrow \pp \Sym^2 F$. However, the proofs in this more general setting are identical and cover other interesting situations, and are in fact notationally less cumbersome.

\subsubsection{Deforming to general principal parts}

\begin{lemma} \label{FirstGoodDef} Let $p \in \Prin (V)$ be a principal part which can be represented by a sum $\sum_{i=1}^m \frac{\sigma_i}{z^{d_i}}$, where each $z_i$ is a local parameter at a point $x_i \in C$ and $\sigma_i$ is a section of the subfibration $\tY \subseteq V$ which is nonzero at $x_i$. Then there exists an analytic family of principal parts $\{ p_s : s \in \Delta \} \subset \cT_d (V)$, where $d = \sum_{i=1}^m d_i$, parameterized by an open disk $\Delta$ around $0 \in \cc$, such that $p_0 = p$ and $p_s$ is a general $\tY$-valued principal part in the sense of Definition \ref{genYvaluedppt} for $s \ne 0$. \end{lemma}

\noindent Note that the $x_i$ in the above statement need not be distinct.

\begin{proof} We follow the approach of \cite[{\S} 2]{CH2}. Choose $d = \sum_{i=1}^m d_i$ distinct complex numbers
\[ \tau_{i,j} : \quad 1 \le i \le m, \quad 1 \le j \le d_i . \]
Let $\Delta$ be a small disk around $0 \in \cc$. For each $s \in \Delta$, let $p_s$ be the principal part
\[ \sum_{i=1}^m \frac{\sigma_i}{(z_i - s \tau_{i,1}) \cdots (z_i - s \tau_{i, d_i})} . \]
Clearly $p_0 = p$, while for $s \ne 0$ the support of $p_s$ consists of $\sum_{i=1}^m d_i = d$ distinct points. 
 If $h^1 ( C, V ) = 0$ then we are done.

Otherwise; using partial fraction decomposition, we see that for $s \ne 0$ we have
\[ p_s \ = \ \sum_{i, j} \frac{\rho_{i, j}}{s} \frac{\sigma_i}{(z_i - s \tau_{i, j})} \]
for nonzero scalars $\rho_{i,j}$ (note that this sum has a removable discontinuity at $s = 0$). Hence, to complete the proof we must show that for $s \ne 0$, the cohomology classes
\begin{equation} \left[ \frac{\sigma_i}{z_i - s \tau_{i, j}} \right] : \quad 1 \le i \le m, \quad 1 \le j \le d_i \label{DeformedClasses} \end{equation}
are in general position, possibly after shrinking $\Delta$.

By hypothesis, the representative $\sigma_i$ is a section of the subfibration $\tY \subseteq V$ over some Zariski neighborhood $U_i$ of $x_i$. Thus it defines a quasiprojective algebraic curve $\sigma_i ( U_i )$ in the total space of $\tY$, and also in $Y \subseteq \pp V$. Consider the image $C_i \subseteq \pp H^1 (C, V)$ of this curve via the map $\psi \colon \pp V \dashrightarrow \pp H^1 (C, V)$. We claim that the representative $\sigma_i$ may be changed if necessary such that $C_i$ is nondegenerate in $\pp H^1 ( C, V )$.

To see this; firstly, recall that the principal part $p$ is determined by the values $\sigma_i \mod m_{x_i}^{d_i} \cdot V_{x_i}$. Shrinking $U_i$, we may assume that $\tY|_{U_i}$ is trivial. Since $U_i$ and the fiber of $\tY$ may be assumed to be affine, and since by hypothesis $Y$ is nondegenerate in $\pp H^1 ( C, V)$, we can if necessary replace $\sigma_i$ with another section $\sigma_i'$ of $\tY|_{U_i}$ with $\sigma_i' \equiv \sigma_i \mod m_{x_i} \cdot V_{x_i}$ such that the image $C_i' \subset \pp H^1 ( C, V )$ of $\sigma_i' (U_i)$ is nondegenerate.

Since $C_i$ is nondegenerate, clearly so is the image in $C_i$ of any open analytic neighborhood of $x_i$ in $U_i$. By construction, the classes (\ref{DeformedClasses}) lie inside a union of such nondegenerate analytic curves in $\pp H^1 ( C , V )$. Thus, 
 shrinking $\Delta$ if necessary, we can assume that these classes are in general position for $s \ne 0$. \end{proof}

\subsubsection{Different representatives for a fixed cohomology class}

Let $p = \sum_{i=1}^m \frac{\sigma_i}{z_i^{d_i}}$ be as in the previous subsubsection, and consider again the deformation $\{ p_s \}$ constructed in Lemma \ref{FirstGoodDef}. We shall now show that if we add more points of $\pp Y$, we can construct a further deformation $\{ p_s' : s \in \Delta \}$ of $p$ whose generic element is a general $Y$-valued principal part satisfying in addition $[ p_s' ] \equiv [ p ] \in H^1 ( C, V )$.

For any $r \ge 1$, choose nonzero $y_1 , \ldots , y_r \in \tY$ lying over distinct points $u_1 , \ldots , u_r$ of $C$ respectively. For $1 \le k \le r$, let $\nu_k$ be a section of $\tY$ near $u_1$ such that $\nu_k ( u_k ) = y_k$. For each $k$, let $w_k$ be a local parameter at $u_k$. By Lemma \ref{alternativePsi}, if $h^1 ( C, V ) \ne 0$ then the cohomology class $\left[ \frac{\nu_k}{w_k} \right]$ defines the image of $y_k$ in $\psi ( Y ) \subseteq \psi ( \pp V ) \subseteq \pp H^1 (C, V)$.

With $p$ and $\{ p_s \}$ as above; 
 since $\psi( Y ) \subseteq \psi (\pp V)$ is nondegenerate in $\pp H^1 ( C, V )$, after perturbing the $y_k$ if necessary, by Lemma \ref{FirstGoodDef} we may assume that for each $s \ne 0$, the $d + r$ cohomology classes
\begin{equation} \left[ \frac{\sigma_i}{z_i - s \tau_{i, j}} \right] : 1 \le i \le m; \; 1 \le j \le d_i \quad \hbox{and} \quad \left[ \frac{\nu_k}{w_k} \right] : 1 \le k \le r \label{GenPosClasses} \end{equation}
are in general position.

We shall require the following easy lemma, whose proof is left to the reader.

\begin{lemma} \label{genpos} Let $H$ be a vector space. Suppose $t \ge \dim (H) + 1$, and let $v_1 , \ldots , v_t \in H$ be in general position. Then any element of $H$ can be written as a linear combination of $v_1 , \ldots , v_t$ in which every coefficient is nonzero. \qed \end{lemma}

\begin{lemma} \label{GoodDef} Assume $d + r > h^1 (C, V)$. Let $p$ and $\{ p_s \}$ and $\frac{\nu_1}{w_1} , \ldots , \frac{\nu_r}{w_r}$ be as above. Then there exist nowhere zero analytic functions $a_{i, j} (s)$ and $b_k (s)$ on $\Delta$ such that the family of principal parts
\[ p'_s \ := \ p_s + \sum_{i=1}^m \sum_{j=1}^{d_i} s \cdot a_{i, j}(s) \cdot \frac{\sigma_i}{z_i - s \tau_{i, j}} + \sum_{k=1}^r s \cdot b_k (s) \cdot \frac{\nu_k}{w_k} \]
satisfies $[ p'_s ] \equiv [ p ]$ for all $s \in \Delta$, and for $s \ne 0$, the principal part $p_s'$ is general $\tY$-valued in the sense of Definition \ref{genYvaluedppt}. \end{lemma}

\begin{proof} We define a map $\Phi \colon \Delta \times \cc^{d + r} \to \Delta \times H^1 ( C, V )$ of affine bundles over $\Delta$, by
\begin{multline*} \Phi \left( s, ( a_{i, j} : 1 \le i \le m ; 1 \le j \le d_i ), ( b_k : 1 \le k \le r ) \right) \ = \\ 
 \left( s, \left[ p_s \right] + \sum_{i, j} s \cdot a_{i, j} \left[ \frac{\sigma_i}{z - s \tau_{i, j}} \right] + \sum_{k = 1}^r s \cdot b_k \left[ \frac{\nu_k}{w_k} \right] \right) . \end{multline*}
For $s = 0$, this is the constant map $\cc^{d+r} \to H^1 ( C, V )$ with value $[ p_0 ] = [ p ]$. On the other hand, if $s \ne 0$ then $\Phi|_s$ is surjective, since the classes (\ref{GenPosClasses}) are nonzero and in general position. Therefore, since $d + r > h^1 ( C, V)$, by Lemma \ref{genpos}, we can choose nowhere zero analytic functions $a_{i, j}(s)$ and $b_k (s)$ such that
\[ \Phi \left( s, ( a_{i, j}(s) : 1 \le i \le m ; 1 \le j \le d_i ), ( b_k (s): 1 \le k \le r ) \right) \ \equiv \ [ p ] \]
for all $s \in \Delta$. Hence, defining the family of principal parts
\[ p_s' \ := \ p_s + \sum_{i=1}^m \sum_{j=1}^{d_i} s \cdot a_{i, j}(s) \cdot \frac{\sigma_i}{z_i - s \tau_{i, j}} + \sum_{k=1}^r s \cdot b_k (s) \cdot \frac{\nu_k}{w_k} , \]
the lemma follows. \end{proof}

\subsubsection{Proof of Proposition \ref{QFcDense}}

\begin{proof}[Proof of Proposition \ref{QFcDense}] Let $E$ be a point of $\qf \setminus \qfc$. We shall prove the proposition by showing that there exists a deformation $\cE_s$ of $E$ in $\qf$ parameterized by a neighborhood $\Delta$ of $0$ in $\cc$ satisfying $\cE_0 = E$ and $\cE_s \in \qfc$ for $s \ne 0$.

The saturation $\bE$ of $E$ is a Lagrangian subbundle of $W$ of degree $-\be \ge -e$. By Lemma \ref{GoodCoords}, we may assume that $W$ is an extension $0 \to F \to W_{p_0} \to F^* \otimes L \to 0$ as in (\ref{W}) for some principal part $p_0 \in \Prin (L^{-1} \otimes \Sym^2 F)$ such that $\bE = \Gamma_0 \cap W_{p_0} \cong \Ker ( p_0 )$. 
 By the proof of \cite[Lemma 2.7]{CH2} (essentially a diagonalization procedure 
 for symmetric matrices over $\Oc$), the principal part $p_0$ can be represented by
\[ \sum_{i = 1}^m \frac{\lambda_i \otimes \eta_i \otimes \eta_i}{z_i^{d_i}} \]
where $z_i$ is a local parameter at a point $x_i \in C$ and $\lambda_i$ a generator for $L^{-1}$ near $x_i$, and $\eta_i$ is a suitable frame element of $F$ near $x_i$; 
 moreover, if $x_{i_1} = \cdots = x_{i_h}$ then $\eta_{i_1} , \ldots , \eta_{i_h}$ are independent at $x_{i_1}$. Furthermore, as $\bE = \Ker ( p_0 )$, we have
\[ \sum_{i=1}^m d_i \ = \ \deg ( F^* \otimes L ) - \deg( \bE ) \ = \ f_0 + n \ell + \be . \]

We claim now that if $\be \ne e$, then we may assume that $E$ is general in $\Elm^{e - \be} ( \bE )$. For; the latter space is completely contained in $\qf$, since $\rank (\bE \cap F) = 0$. As $\Elm^{\be - e} ( \bE )$ is irreducible, if a general point belongs to the closure of $\qfc$ in $\qf$, then in fact every point does. Therefore, we may assume that
\[ \bE / E \ \cong \ \bigoplus_{k=1}^{e - \be} \cO_{u_k} , \]
for $e - \be$ distinct points $u_k \in C$ lying outside $\Supp ( p )$. Thus for $1 \le k \le e - \be$, there exists a local coordinate $w_k$ centered at $u_k$ and frame elements $\zeta_k$ and $\mu_k$ for $F$ and $L^{-1}$ respectively near $u_k$, such that $[E \to F^* \otimes L] \in \Elm^{f_0 + n\ell + e} ( F^* \otimes L )$ satisfies
\[ E \ = \ \Ker (p_0) \cap \Ker \left( \sum_{k=1}^{e - \be} \frac{\mu_k \otimes \zeta_k \otimes \zeta_k}{w_k} \right) \ = \ \Ker \left( p_0 + \sum_{k=1}^{e - \be} \frac{\mu_k \otimes \zeta_k \otimes \zeta_k}{w_k} \right) . \]
Here, as usual, we view the principal parts as $\Oc$-linear maps $F^* \otimes L \to \sPrin (F)$.

We now specialize the results of the previous subsections to the present situation. Set $F_1 = F$ and $F_2 = F^* \otimes L$, and $V = L^{-1} \otimes \Sym^2 F$. Let $Y$ be the relative Segre embedding $\pp F \hookrightarrow \pp \Sym^2 F$, and set $\sigma_i = \lambda_i \otimes \eta_i \otimes \eta_i$ and $\nu_k = \mu_k \otimes \zeta_k \otimes \zeta_k$. Also, $d = \be + f_0 + n\ell$ and $r = e - \be$, so $d + r = e + f_0 + n\ell$.

Continuing with this input, for each $s \in \Delta$, we set

\begin{multline*} p_s' \ := \ \sum_{i=1}^m \frac{\lambda_i \otimes \eta_i \otimes \eta_i}{(z_i - s \tau_{i,1}) \cdots (z_i - s \tau_{i, d_i})} + \sum_{i=1}^m \sum_{j=1}^{d_i} s \cdot a_{i, j}(s) \cdot \frac{\lambda_i \otimes \eta_i \otimes \eta_i}{z_i - s \tau_{i, j}} \\ + \sum_{k=1}^r s \cdot b_k (s) \cdot \frac{\mu_k \otimes \zeta_k \otimes \zeta_k}{w_k} , \end{multline*}
precisely as constructed in Lemma \ref{GoodDef}, and we consider the family $\{ \cE_s : s \in \Delta \}$ of elementary transformations of $F^* \otimes L$ given by $\cE_s = \Ker ( p'_s ) \subset F^* \otimes L$.

By construction of $\{ p_s' \}$, for $s \ne 0$ the sheaf $\cE_s$ has degree $-e$. 
 By Lemma \ref{GoodDef}, for $s \ne 0$ the principal part $p'_s$ is general $Y$-valued so, by Lemma \ref{vanishinglemma} the sheaf $\cE_s$ satisfies property (ii) of $\qfc$ in Definition \ref{DefnQFc}.

Furthermore, by hypothesis, $e \ge e_1 (W)$ where $e_1 (W)$ is as defined in (\ref{eoneW}). Therefore, by (\ref{eoneW}) and the discussion before it, we have
\[ d + r \ = \ e + f_0 + n\ell \ > \ h^1 ( C, L^{-1} \otimes \Sym^2 F ) . \]
Hence, by the statement of general position in Lemma \ref{GoodDef} the cohomology classes
\[ \left[ \frac{\lambda_i \otimes \eta_i \otimes \eta_i}{z_i - s \tau_{i, j}} \right] : 1 \le i \le m; 1 \le j \le d_i \quad \hbox{and} \quad \left[ \frac{\mu_k \otimes \zeta_k \otimes \zeta_k}{w_k} \right] : 1 \le k \le r \]
span $H^1 ( C, L^{-1} \otimes \Sym^2 F )$. Thus $\cE_s$ also has property (iii) of Definition \ref{DefnQFc} for $s \ne 0$.

Next, let $\cW \to \Delta \times C$ be the family of extensions associated to $\{ p'_s : s \in \Delta \}$ as discussed in {\S} \ref{FamiliesPptExt}. As each $p'_s$ is symmetric, as before the standard symplectic form (\ref{standardsympform}) restricts to a symplectic structure on each $\cW_{p'_s}$ by \cite[Criterion 2.1]{Hit1}.

Since by Lemma \ref{GoodDef} we have $[ p'_s ] \equiv [ p ] = \delta (W)$, in fact $\cW_s \cong W$ for all $s \in \Delta$. Hence by Lemma \ref{FamilyLifting}, the family
\[ \{ \cE_s = \Ker ( p'_s ) \subset \pi_C^* (F^* \otimes L) : s \in \Delta \} \]
in fact defines a family of saturated subsheaves of $W$. As $\cE_s \subset \sRat (F^* \otimes L)$ and the latter is isotropic with respect to (\ref{standardsympform}), in particular each $\cE_s$ is isotropic.

Now $\{ \cE_s \}$ is flat only over $\Delta \setminus \{ 0 \}$, as $\Ker ( p_s )$ has degree $-e$ for $s \ne 0$ but $\cE_0 = \bE$ has degree $-\be$. We replace the central member $\cE_0$ by the flat limit $\wcE_0$ of $\{ \cE_s : s \neq 0 \}$ in $\LQ$. Then $\wcE_0$ is a full rank subsheaf of $\bE = \Ker (p)$. We will have finished if we can show that $\wcE_0 = E$. Now for $s \ne 0$, each $\cE_s$ is contained in
\begin{equation} \Ker \left( \sum_{k=1}^{e - \be} s \cdot \frac{\mu_k \otimes \zeta_k \otimes \zeta_k}{w_k} \right) \ = \ \Ker \left( \sum_{k=1}^{e - \be} \frac{\mu_k \otimes \zeta_k \otimes \zeta_k}{w_k} \right) . \label{constantET} \end{equation}
Thus the limit $\wcE_0$ is also contained in (\ref{constantET}). As moreover $\wcE_0 \subseteq \bE \cong \Ker ( p_0' ) = \Ker ( p_0 )$, we have
\[ \wcE_0 \ \subseteq \ \Ker \left( \sum_{k=1}^{e - \be} \frac{\mu_k \otimes \zeta_k \otimes \zeta_k}{w_k} \right) \cap \Ker ( p_0' ) , \]
which is exactly $E$. As $\deg ( \wcE_0 ) = -e$ by flatness, we have $\wcE_0 = E$. \end{proof}

\begin{remark} \label{geominterp} The deformation $\{ p_s' \}$ is most naturally understood from the point of view of secant geometry. For simplicity, assume that $L = \Oc$ and $\psi \colon \pp F \dashrightarrow \pp H^1 (C, \Sym^2 F )$ is generically an embedding and that $E \subset F^*$ is a general elementary transformation corresponding to $e+f > h^1 ( C, \Sym^2 F )$ general points of $\pp F$. By \cite[Lemma 2.10 (i)]{CH2}, if $E$ is nonsaturated in $W$ then $\delta(W)$ lies on the secant spanned by $(\be+f) < (e+f)$ of these points. Moving inside the family $p'_s$ corresponds to perturbing the linear combination of the points defining $\delta(W)$ to be nonzero at all $e+f$ points, so as to obtain a principal part supported at exactly $e+f$ points, so defining a saturated isotropic subsheaf of degree $-e$. \end{remark}

\end{document}